\documentclass[12pt]{article}

\usepackage[centertags,sumlimits]{amsmath}
\usepackage{amssymb}
\usepackage{amsthm}

\newtheorem{thm}{Theorem}[section]
\newtheorem{lemma}[thm]{Lemma}
\newtheorem{prop}[thm]{Proposition}
\newtheorem{cor}[thm]{Corollary}
\newtheorem{rem}[thm]{Remark}

\newtheorem*{defn}{Definition}
\newtheorem*{remark}{Remark}

\numberwithin{equation}{section}

\columnwidth\textwidth

\newcommand{\ca}{{\cal A}}    \newcommand{\cb}{{\cal B}}   
\newcommand{\ce}{{\cal E}}    \newcommand{\cf}{{\cal F}}   
       
    \newcommand{\cn}{{\cal N}}   \newcommand{\cp}{{\cal P}}
\newcommand{\calr}{{\cal R}}     \newcommand{\ct}{{\cal T}}
\newcommand{\cu}{{\cal U}}       \newcommand{\cv}{{\cal V}}
       \newcommand{\cq}{{\cal Q}}

   \newcommand{\N}{\mathbb{N}}  
\newcommand{\al}{{\alpha}}

\begin{document}

\noindent
{\Large\bf On the quasi-regularity of non-sectorial Dirichlet forms by processes having
the same polar sets}\\[2mm]

\noindent {\bf Lucian Beznea}\footnote{"Simion Stoilow" Institute
of Mathematics  of the Romanian Academy, P.O. Box \mbox{1-764,}
RO-014700 Bucharest, Romania. E-mail: lucian.beznea@imar.ro} {\bf
and Gerald Trutnau}\footnote{Seoul National University, Department
of Mathematical Sciences and Research Institute of Mathematics, San56-1 Shinrim-dong Kwanak-gu, Seoul
151-747, South Korea. E-mail: trutnau@snu.ac.kr}

\vspace{7mm}

\noindent {\small{\bf Abstract.} 
We obtain a criterion for the quasi-regularity of generalized (non-sectorial) Dirichlet forms,
which extends the result of P.J. Fitzsimmons on the
quasi-regularity of (sectorial) semi-Dirichlet forms.
Given the right (Markov) process
associated to a semi-Dirichlet form, we present sufficient
conditions for a second right process to be a standard one, having
the same state space.  The above mentioned quasi-regularity criterion is then an application.  
The conditions are expressed in terms of the associated capacities, nests of compacts, polar sets, and quasi-continuity.
A second  application is   on  the quasi-regularity of the generalized Dirichlet forms obtained by
perturbing  a semi-Dirichlet form with kernels.\\

}


\section{Introduction}\label{1}
The theory of Dirichlet forms is a powerful tool in the study of Markov processes, since it
combines different areas of mathematics such as probability, potential, and semigroup theory, as well as 
the theory of partial differential equations (see monographs \cite{fot}, \cite{mr} and references therein). 
For instance, the classical energy calculus in combination with the potential theory of additive functionals 
allows to 
obtain an extension of It\^o's formula for only
weakly differentiable functions, i.e. functions in the domain of the form. 
This celebrated extension of It\^o's formula where the martingale and the possibly unbounded variation drift part 
are controlled through the energy is 
well-known as {\it Fukushima's decomposition} of additive functionals (see e.g. \cite[Theorem 5.2.2]{fot}). \\
Until recent years the applicability of Dirichlet form theory was limited
to symmetric Markov processes (see \cite{fot}) or, more generally, to Markov processes satisfying a 
sector condition (cf. \cite{mr}; see also sections 7.5  and  7.7  from  \cite{BeBo 04}  
for the connections with the right processes). 
Within the theory of {\it generalized} (non-sectorial) {\it Dirichlet forms} (see \cite{St1}, 
and \cite{Tr1} for the associated stochastic calculus)
this limitation has been overcome since in this generalized framework only the existence 
of a positive measure $\mu$ is required for which the transition semigroup of the Markov process 
operates as a $C_0$-semigroup of contractions on $L^2(\mu)$. 
In particular, as no sector condition has to be verified, the 
theory of generalized Dirichlet forms is robust and well-suited for far-reaching perturbation methods.\\
A central analytic property in the theory of symmetric, sectorial, and non-sectorial Dirichlet forms 
is the quasi-regularity of the forms, because only such forms can be associated with a nice Markov process, 
i.e. a {\it standard process} (see \cite[IV. Theorem 3.1]{St1}). 
Moreover, after review of \cite{Tr1} and the present paper one can recognize that a 
stochastic calculus for generalized Dirichlet forms can be developed by only 
assuming the standardness of an associated process and the existence of an excessive measure for the underlying $L^2$-space 
of the form (such a measure is called an {\it excessive reference measure}). 
On the other hand it was shown in \cite{RoTr} that any right continuous Markov process 
can be related to a generalized Dirichlet form with excessive reference measure. 
We were hence led to the following question:
Given a right Markov process $X$ on a fairly general state space. Under which 
additional analytic condition is $X$ a standard process ? 
This question has first  been investigated in \cite{Fi 01} under the additional condition that 
$X$ is associated with a semi-Dirichlet from (see \cite{MaOvRo}). 
A different approach and several extensions have been developed  in sections 3.7 and 7.7 from \cite{BeBo 04}.
Moreover as an application in \cite{Fi 01} 
the theory of Revuz measures in the semi-Dirichlet 
context is developed and it is shown that quasi-regularity is invariant under time change. 
In particular, if there is an {\it excessive} reference measure 
then the developed theory of Revuz measures can be related to the context 
of classical energy (see \cite[(4.19) Remarks]{Fi 01}).  Hence from the 
point of view of applications it is interesting to investigate the question which additional analytic condition 
leads to the standardness in the case of generalized Dirichlet forms or likewise in the case of any right Markov process. 
This is what we do in subsections \ref{3.1} and \ref{2} respectively.\\
The definition of standard process is quite abstract and technical but one can say in general that the standardness property 
and there mainly the quasi-left continuity connects the analytic   
with the probabilistic potential  theory. For a right Markov process that is not quasi-left continuous 
only capacity zero sets are polar and not vice versa (see Remark \ref{quasileftandcap}(a) for more on this). \\ 
Our main application is an extension of a result in \cite{Fi 01} to the case of non-sectorial generalized Dirichlet 
forms and can roughly and abridged be stated as follows. 
Given two right Markov processes on the same state space  
from which one is associated to a semi-Dirichlet form $\ca$ and from which the other one is associated to a 
generalized Dirichlet form whose sectorial part is given by $\ca$. 
Then the second process is standard if the 
capacities of both forms are equivalent (cf. Remark \ref{exgendf}$(b)$, Theorem \ref{gdfstandard}$(i)$). 
Of course, if the two  bilinear forms coincide we may assume that both right Markov processes are the same. 
Then the condition on the capacity is trivially satisfied and 
we obtain the original result from \cite{Fi 01}, i.e. that a right Markov process that is 
associated with a semi-Dirichlet form is automatically a standard (Markov) process. 
For the mathematically precise formulation of our main result and why we recover the case of semi-Dirichlet 
forms see the paragraph right in front of Theorem \ref{gdfstandard} and the theorem itself.
It is important to mention 
that since we compare the second process with an \lq\lq elliptic\rq\rq\ sectorial process associated with a semi-Dirichlet form, 
we are not able to handle the time-dependent case. We guess that the time-dependent case can be handled by comparing the 
second process to a standard process associated 
with a time-dependent Dirichlet form (see \cite{o}, \cite{RuTr}). This \lq\lq guess\rq\rq\ might be subject of future investigations.\\
Once the standardness is shown,  as in the original paper \cite{Fi 01}, 
we can derive the quasi-regularity 
of the associated Dirichlet form (see Theorem \ref{gdfstandard}(ii)). 
The tightness, i.e. the existence of a nest of compacts, is automatically satisfied 
if we only slightly concretize the state space (see Remark \ref{quasileftandcap} $(b)$ and $(c)$), 
and the special property is in contrast to \cite{mr}, and \cite{St1}, not used in order to show 
that the process resolvent applied to bounded $L^2$-functions is quasi-continuous, i.e. that the process is
properly associated in the resolvent sense with the form  (see Remark \ref{gdfqr}). 
For the quasi-regular generalized Dirichlet form and the associated standard process all results from \cite{St1}, 
and \cite{Tr1}  will be available (cf. first paragraph of section \ref{3}).\\
In subsection \ref{3.2} we present an application where the resolvent of a right 
Markov process is explicitly given as the perturbation of the resolvent of a quasi-regular semi-Dirichlet form. 
Typical examples are given through perturbations with the $\beta$-potential kernel of a continuous additive functional (see \cite{BeBo 09}) and 
by a potential theoretical approach to measure-valued
discrete branching (Markov) processes (see \cite{BeOp11}).
First, in subsection \ref{3.2} we develop general conditions for the perturbed resolvent to be associated with a standard process $X$
(see Proposition \ref{capacitiesperturbation}(ii), and (iii)).  
Then, under the {\it absolute continuity condition},  in subsection \ref{3.2.1} 
we construct explicitly a quasi-regular generalized Dirichlet form that is 
properly associated in the resolvent sense with $X$ (see Corollary \ref{GDF2}). 
As in \cite{RoTr} the generalized Dirichlet form is constructed directly from its generator 
which is also the generator of the underlying Markov process, 
i.e. the $L^2$-infinitesimal generator of the transition 
semigroup $(p_t^{\beta})_{t\ge 0}$ of the Markov process (cf. (\ref{semigroup}) and paragraph below). 
In this application we would like to emphasize two points. First, we do not obtain as usual the standard process 
by showing regularity properties of the form. We rather show that the form is associated to a standard process 
and from this we derive its quasi-regularity. Second, this application shows that the theory of 
generalized Dirichlet forms is well-suited for far-reaching perturbation methods as we do not need any sector condition to be verified after the perturbation. 
\\
Our mentioned applications from section \ref{3} are essentially based on the results that we develop in section \ref{2} (see Theorem \ref{thm2.4}) 
about the sufficient conditions which ensure that the standardness  property is transfered from a right (Markov) process to a second one.
We use several analytic and probabilistic potential theoretical tools 
(as  implemented in \cite{BeBo 04}, \cite{BeBo 05}, and \cite{BeBoRo 06}; see also \cite{BeRo 11} for further applications) 
like  the capacities associated to a right process and their  tightness property,
the quasi-continuity, Ray topologies and compactifications, and the fine topology. 
We complete the second section with the treatment of the weak duality hypothesis frame  (Theorem \ref{thm2.5}).

.

\section{Standardness properties}\label{2}

Let $X =(\Omega, \cf, \cf_t, X_t , \theta_t, P^{x} )$ be a Borel  right (Markov) process
whose state space is a Lusin topological space
$(E, \ct)$, and let $\cu=(U_\alpha)_{\alpha>0}$ be its associated sub-Markovian resolvent.

Let $\beta> 0$ be arbitrary. Recall that the Borel $\sigma$-algebra $\cb$ on $E$ is generated by the set
$\ce (\cu_\beta)$ of all   $\cb$-measurable $\cu_\beta$-excessive functions,
where $\cu_\beta := (U_{\beta +\alpha})_{\alpha>0}$. We further recall that a universally measurable function $f$ is said to be
$\cu_\beta$-excessive if $\alpha U_{\beta +\alpha}f \nearrow f$ pointwise as $\alpha \nearrow \infty$.

Let $\mu$ be a $\sigma$-finite measure on $(E, \cb)$. We say that the right
process $X$ is {\it $\mu$-standard} if for one (and hence all) finite measures $\lambda$ which are equivalent to $\mu$  it possesses left
limits in $E$ $P^\lambda$-a.e. on $[0, \zeta )$ and for every
increasing sequence $(T_n)_n$ of stopping times with $T_n\nearrow
T$ we have $X_{T_n}\longrightarrow X_T$ $P^\lambda$-a.e. on $\{T<
\zeta\}$, $\zeta$ being the lifetime of $X$. If in addition
$\cf_T^\lambda =\bigvee_n \cf_{T_n}^\lambda$ then $X$ is called
{\it $\mu$-special standard} (cf. Section 16 in \cite{GeSh
84}).

A {\it Ray cone} associated with $\cu$ is a convex cone $\calr$ of
bounded$\cb$-measurable, $\cu_\beta$-excessive  functions such that:
\begin{itemize}
	\item[$\bullet$] The cone $\calr$  contains the positive constant functions and is min-stable.
	\item[$\bullet$] $U_\beta ((\calr -\calr)_+) \subset \calr$ and
$U_\alpha (\calr) \subset \calr$  for all $\alpha >0 .$
	\item[$\bullet$] The cone $\calr$ is separable  with respect to the uniform norm.
	\item[$\bullet$] The $\sigma$-algebra on $E$  generated by $\calr$  coincides with  $\cb .$
\end{itemize}

One can show that for every countable set $\ca$ of bounded $\cb$-measurable,
$\cu_\beta$-excessive  functions
there exists a Ray cone including $\ca$.

The topology $\ct_\calr$ on $E$ generated by a Ray cone $\calr$
(i.e. the coarsest topology on $E$ for which every function from
$\calr$ is continuous) is called the {\it Ray topology} induced by
$\calr$.

A Lusin topology on $E$ is called {\it natural} (with respect to
$\cu$) if its Borel $\sigma$-algebra is precisely $\cb$ and it is
smaller than the {\it fine topology} on $E$ (with respect to
$\cu$). We recall that the {\it fine topology} with respect to $\cu$ is the
topology on $E$ generated by all $\cu_\beta$-excessive
functions. The initial topology $\ct$ as well as any Ray topology are
natural. Further note that $\cu$ is the resolvent of a right process
with respect to any natural topology (cf. \cite{BeBo 04}). \\

If $\beta>0$, then for all $u\in \ce(\cu_\beta)$ and every subset
$A$ of $E$ we consider the function
$$
R_\beta^A u := \inf \{ v \in \ce(\cu_\beta)| \, v\geq u \mbox{ on } A \},
$$
called the $\beta$-order {\it reduced function} of $u$ on $A$.
It is known that if $A\in \cb$ then $R^A_\beta u$  is universally
$\cb$-measurable and if moreover $A$ is finely open and $u\in
p\cb$ then $R_\beta^A u\in p\cb$. If $A=E$ we simply write $R_{\beta}u$ instead
of $R_{\beta}^Eu$.\\

Let   $\lambda$ be a finite measure on $(E, \mathcal{B})$. We also fix
a strictly positive, bounded $\cu_\beta$-excessive function
$p_o$  of the form $p_o=U_\beta f_o ,$ with $f_o\in p\cb,$
$0<f_o\leq 1$.

Since  $\cu$ is the resolvent of a right process $X$, the
following fundamental result of G. A. Hunt holds for all $A\in
\cb$ and $u\in \ce(\cu_\beta)$:
$$
R_\beta^A u(x)= E^x(e^{-\beta D_A} u (X_{D_A}))
$$
where $D_A$ is the entry time of $A$, $D_A=\inf \{ t\geq 0 |$ $ X_t\in A \}$;
 see e.g. \cite{DeMe 87}.

It turns out  that the functional $M\longmapsto c^\beta_\lambda(M) ,$ $M\subset E ,$ defined by
$$
c^\beta_\lambda(M)=\inf\{\lambda(R_\beta^G p_o) | \,  G\in\mathcal{T},\;M\subset G \}
$$
is a Choquet capacity on  $(E, \mathcal{T})$.

The capacity $c^\beta_\lambda$ on $(E, \mathcal{T})$
is called {\it tight} provided that there exists an increasing
sequence $(K_n)_n$ of $\mathcal{T}$-compact sets such that
$$
\inf_n c^\beta_\lambda(E\setminus K_n)=0
$$
(or equivalently  $\inf_n R_\beta^{E\setminus K_n}p_o=0$
$\lambda$-a.e.) which is also equivalent to
$$
P^\lambda (\lim_n D_{E\setminus K_n} < \zeta )=0.
$$
In particular, if  the capacity $c^\beta_\lambda$ on $(E, \mathcal{T})$ is tight for
one $\beta>0$ then this happens for all  $\beta>0$. Similarly, the following assertion holds:

\begin{rem}\label{rem2.1}   %Remark 2.1
 If $(G_n)_n$ is a decreasing sequence of
$\ct$-open sets such that there exists $\beta>0$ with  $\inf_n
c_\lambda^\beta(G_n)=0$ then
%the above
the equality holds for all $\beta>0$.
\end{rem}

A set $M$ is called {\it $\lambda$-polar} with respect to $\cu$ provided that there
exists $A\in \cb$, $M\subset A$, such that $T_A = \infty $
$P^\lambda$-a.e., where $T_A$ is the hitting time of $A$,
$T_A=\inf \{ t > 0 |$ $ X_t\in A \}$.

A real valued function $u\in \ce(\cu_\beta)$ is called {\it regular} provided that
for every sequence $(u_n)_n$ in $\ce(\cu_\beta)$, $u_n\nearrow u$, we have
$\inf_n R_\beta (u-u_n)=0$;
see \cite{BeBo 04} for more details on regular excessive functions.
It is known that (see e.g. \cite{DeMe 87}) if the process $X$ is transient, then a bounded function $u\in\ce(\cu)$
is regular if and only if there exists a continuous additive functional having $u$ as potential function.
A real valued  $\cu_\beta$-excessive function $u$ is called  $\lambda${\it -regular} with respect to $\cu_{\beta}$, provided there exists
a regular $\cu_\beta$-excessive function $u'$
such that $u=u'$ $\lambda\circ U_\beta$-a.e.

\begin{prop}\label{prop2.0}  % Proposition 2.2
The following assertions hold.
\begin{itemize}
	\item[(i)] $A\in \cb$ is $\lambda$-polar with respect to $\cu$ and
$\lambda$-negligible if and only if $\lambda(R^A_\beta p_o)=0$.
Consequently if a Borel set is of $c^\beta_\lambda$-capacity zero
then it is $\lambda$-polar and $\lambda$-negligible
	\item[(ii)] Assume that one of the following two conditions is
satisfied:
	
\begin{itemize}
	\item[(ii.a)] The topology $\ct$ is a Ray one.
	\item[(ii.b)] Every $\cu_\beta$-excessive function dominated by $U_\beta f_o$ is $\lambda$-regular.
\end{itemize}
\end{itemize}

\noindent Then $c^\beta_\lambda (A)=\lambda(R^A_\beta p_o)$ for all $A\in \cb$ and in
particular the sets which are $\lambda$-negligible and
$\lambda$-polar are precisely those having $c^\beta_\lambda$-capacity
zero. If condition $(ii.b)$ holds then  the capacity
$c^\beta_\lambda$ is tight in any natural topology.
\end{prop}

\begin{proof}
$(i)$ The first statement in $(i)$ is immediate from the
definitions. For the second we present two proofs, a direct one for the convenience of the reader, and a shorter alternative proof based on known results:\\
{\it Direct proof:} If $c^\beta_\lambda (A)=0$ then there are open sets $U_k\supset A$ with $c^\beta_\lambda (U_k)\to 0$ as $k\to \infty$.
Define $F_k:=E\setminus (\cap_{l\le k}U_l)$. Then $A\subset E\setminus F_k$ for all $k$, hence
\begin{eqnarray*}
P^{\lambda}(T_A<\infty) & \le & P^{\lambda}(D_A<\infty)\\
& \le & \lim_k P^{\lambda}(D_{E\setminus F_k}<\infty)=P^{\lambda}(\lim_k D_{E\setminus F_k}<\infty).
\end{eqnarray*}
We have $c^\beta_\lambda (E\setminus F_k)\le c^\beta_\lambda (U_k)\to 0$. Thus by Lebesgue
$$
0=\lim_k c^\beta_\lambda (E\setminus F_k)=\int_E E^x\left (\int_{\lim_k D_{E\setminus F_k}}^{\infty} e^{-\beta t}f_o(X_t)dt\right )\lambda(dx),
$$
and so $P^{\lambda}(\lim_k D_{E\setminus F_k}<\infty)=0$. Consequently, $A$ is $\lambda$-polar and $\lambda$-negligible.\\
{\it Alternative proof:} (a) If the topology $\ct$ is a Ray one then the assertion is a direct consequence of Proposition 1.6.3 in \cite{BeBo 04}.\\
(b) If the topology $\ct$ is only natural, then there exists a Ray topology $\ct_R$ which is finer than the given topology $\ct$ (we use this procedure in the proof of (ii)). 
Therefore, if a set $A$ has zero capacity w.r.t. the capacity constructed using $\ct$, then $A$ is of zero capacity if we replace $\ct$ by $\ct_R$, hence it is $\lambda$-polar by (a). 

$(ii)$ If $(ii.a)$  holds then the assertion follows from
Proposition 1.6.3 in \cite{BeBo 04}.

Assume that $(ii.b)$ is verified.   Using Proposition 2.1 in
\cite{BeBo 05} we may consider a Ray cone $\calr$ (formed by
$\cu_\beta$-excessive functions) such that the topology
$\ct_\calr$ generated by $\calr$ is finer than the given  natural
topology $\ct$. By $(ii.b)$ and Theorem 3.5.2 in \cite{BeBo 04} it
follows that the capacity $c^\beta_\lambda$ is tight in the Ray
topology $\ct_\calr$ and therefore also in the topology $\ct$.

Let $A\in \cb$ and $\varepsilon
>0$. We consider a $\ct_\calr$-compact set $K$ such that
$c^\beta_\lambda(E\setminus K) < \frac{\varepsilon}{2}$. By the
above considerations
and the definition of $c^\beta_\lambda$
there exists a $\ct_\calr$-open set $G$ such
that $A\subset G$ and $c^\beta_\lambda(G) <  \lambda(R_\beta^A
p_o) + \frac{\varepsilon}{2}$. Let $G_o= (G\cap K)\cup (E\setminus
K)$. Then $A\subset G_o$ and since $\ct_\calr |_K= \ct |_K$, it
follows that the set $G_o$ is $\ct$-open. We have
$c^\beta_\lambda(G_o)\leq c^\beta_\lambda(G\cap K) +
c^\beta_\lambda(E\setminus K)\leq   \lambda(R_\beta^A p_o)+
\varepsilon$ and we conclude that, considering $c^\beta_\lambda$
as a capacity on $(E, \ct)$, we have $c^\beta_\lambda
(A)=\lambda(R_\beta^A p_o).$
\end{proof}

\begin{rem}\label{quasileftandcap}  % Remark 2.3

\begin{itemize}
	\item[(a)] For the converse of the second statement in Proposition \ref{prop2.0}(i)
in general one needs at least the quasi-left continuity of $X$.
For instance, if $X$ is $\lambda$-standard, then the converse holds. In this sense 
a $\lambda$-standard process connects the analytic capacity related to excessive functions and the process 
capacity related to polar sets.
	\item[(b)] Let $\mu$ be a $\sigma$-finite measure on $E$. 
Recall that the right process  $X$ is 
said to be {\rm $\mu$-tight}  provided  
there exists an increasing sequence $(K_n)_{n\in \N}$ of $\ct$-compact (metrizable) sets in $E$ 
such that $P^{\mu}(\lim_{n\to \infty}D_{E\setminus K_n}<\zeta)=0$. 
In particular, if $\lambda$ is a finite measure on $E$,  then the  $\lambda$-tightness of $X$ is equivalent 
to the tightness of  $c^\beta_\lambda$  for some $\beta>0$ (see explanations right before Remark \ref{rem2.1}).
\item[(c)] Suppose that the right process $X$ has $P^{\mu}$-a.e. left limits in $E$ up to
$\zeta$  and that $E$ is a metrizable
Lusin space. Then $X$ is automatically
$\mu$-tight (see \cite[IV. Theorem 1.15]{mr}).
 In particular,  $X$ is  automatically $\mu$-tight if $X$ is 
$\mu$-standard (since the existence of the left limits up to
$\zeta$ $P^{\mu}$-a.e. is part of the definition of the $\mu$-standardness).

\end{itemize}
\end{rem}

The main argument in the proof of the next result is a modification of the proof of
$ii) \Longrightarrow i)$ from Theorem $1.3$ in \cite{BeBo 05}.

\begin{prop}\label{prop2.1}  %Proposition 2.4
Assume that the topology $\ct$ is generated by a Ray cone $\mathcal{R}$
and let $\mu$  be a $\sigma$-finite measure on $(E, \mathcal{B})$.
Then the following assertions hold.
\begin{itemize}
	\item[(i)] If the process $X$ is  $\mu$-tight,  then it  is $\mu$-standard.
	\item[(ii)] If  the $\beta$-subprocess of $X$ is  $\mu$-standard for some $\beta>0$  
then  $X$ is also   $\mu$-standard.
\end{itemize}
\end{prop}

\begin{proof} 
 $(i)$  Suppose that $X$ is  $\mu$-tight and  let $\lambda$ be a finite measure on $E$ which is equivalent with $\mu$.
 By assertion $(b)$ of Remark \ref{quasileftandcap} the capacity $c^\beta_\lambda$ is tight
 and let $(K_n)_n$ be an increasing sequence of  $\ct$-compact subsets
of $E$ such that  
$\inf_nR_\beta^{E\setminus K_n}p_o=0$
$\lambda$-a.e. 
We denote by  $K$ the Ray compactification of $E$
with respect to $\mathcal{R}$ (see, e.g., section 1.5 in \cite{BeBo 04}). 
Since for  every  $u\in\mathcal{R}$
the process  $(e^{-\beta t} u(X_t))_{t\geq 0}$ is a bounded
right continuous supermartingale with respect to the filtration
$(\mathcal{F}_t)_{t\geq 0}$ it follows that (cf. \cite{DeMe 87})
this process has left limits $P^\lambda$-a.e. Since the Ray cone
$\mathcal{R}$ is separable with respect to the uniform norm, it
follows that the process $(X_t)_{t\geq 0}$ has left limits in $K$
$P^\lambda$-a.e. From  $\lim_nR_\beta^{E\setminus K_n}p_o=0$
$\lambda$-a.e. and $\lambda(R_\beta^{E\setminus K_n}p_o)=$
$E^\lambda(\int_{T_{E\setminus K_n}}^\zeta e^{-\beta t} f_o(X_t) dt)$
we deduce that  $\sup_nT_{E\setminus K_n}\geq\zeta$
$P^\lambda$-a.e. 
Hence for every $\omega\in \Omega$ with
$T_{E\setminus K_n}(\omega)<\zeta(\omega)$ we have $X_t(\omega)\in
K_n$ provided that  $t< T_{E\setminus K_n}(\omega)$ and so
$X_{t-}(\omega)\in K_n$. Consequently the process  $(X_t)_{t\geq
0}$ has left limits in $E$ $P^\lambda$-a.e. on $[0, \zeta)$. By
Theorem (48.15) in \cite{Sh 88}  we get that  the ($0$)-process is
$\lambda$-standard.

$(ii)$ If a $\beta$-subprocess of $X$  is  $\mu$-standard, then from Remark
\ref{quasileftandcap} $(c)$  (see also  \cite{LyRo 92} and \cite{BeBo 05})  
it follows that $X$ is $\mu$-tight and by  $(i)$  we get that $X$ is $\mu$-standard.
\end{proof}

A set $M\in \cb$ is called {\it $\lambda$-inessential} (with
respect to $\cu$) provided that it is $\lambda$-negligible and
$R^M_\beta 1=0$ on $E\setminus M$.

\begin{rem}\label{rem2.2}   %Remark 2.6
If $M\in \cb$ is a $\lambda$-inessential set then we may consider
the restrictions $X|_F$ of the process $X$ to the "absorbing set"
$F:=E\setminus M$. Note that $U_\beta (1_M)=0$ on $F$ and the
resolvent associated with $X|_F$ is precisely the restriction
$\cu|_F$ of $\cu$ to $F$. The following assertions hold:
\begin{itemize}
	\item[(a)] $\ce(\cu_\beta|_F)=\ce(\cu_\beta )|_F$. In particular the fine
topology on $F$ with respect to $\cu|_F$ is the trace on $F$ of
the fine topology on $E$ with respect to $\cu$. The process $X$ is
$\lambda$-standard if and only if $X|_F$ is $\lambda|_F$-standard.
	\item[(b)] {\bf Trivial modification.} We consider the {\rm trivial modification of $\cu$
on $M$} (see e.g.  \cite{BeBo 04} and \cite{BeBoRo 06}), namely the sub-Markovian resolvent
$\cu'=(U'_\alpha)_{\alpha>0}$ on $(E, \cb)$
defined by:
$$
U'_\alpha f= 1_E U_\alpha (f 1_F) + \frac{1}{\alpha} f 1_M, \quad \alpha>0, \ f \in \cb.
$$
Then $\cu'$ is the resolvent of a right (Markov) process with
state space $E$. A function $u\in p\cb$ belongs to
$\ce(\cu'_\beta)$ if and only if $u|_F \in \ce(\cu_\beta|_F)$.
Consequently by $(a)$ we have: a subset $\Gamma$ of $E$ is finely
open with respect to $\cu'$ if and only if there exists
a finely open set $\Gamma_o$ with respect to $\cu$, such that $\Gamma\cap F=
\Gamma_o \cap F$. In particular every topology on $E$ which is
natural with respect to $\cu$ is also natural with respect to
$\cu'$.
\end{itemize}

\end{rem}

Let  $\cv=(V_\al)_{\al>0}$ be a second sub-Markovian resolvent of
kernels on $E$ and assume that it is also the resolvent of a right
(Markov) process $Y$ with state space $E$.
Let $\mu$ be a $\sigma$-finite Borel measure on the Lusin space $(E, \ct)$ that has full support.
We suppose that there is a semi-Dirichlet form
$(\ca, D(\ca))$ on $L^2(E, \mu)$ 
with associated $L^2(E,\mu)$-resolvent $(H_{\alpha})_{\alpha>0}$ (see \cite{MaOvRo}).
We further assume that $V_{\alpha} f$ is a $\mu$-version of $H_{\alpha} f$
for any $f\in L^2(E,\mu)$, and $\alpha>0$, i.e. that the right process $Y$ is associated with $(\ca, D(\ca))$.
In particular
$$
H_{\alpha}(L^2(E, \mu))\subset
D(\ca) \mbox{ densely,}
$$
and
$$
\ca_{\alpha}(H_{\alpha}f, u)=(f,u)_{L^2(E, \mu)}
$$
for all $\alpha >0$, $f \in L^2(E, \mu)$, and $u\in  D(\ca)$, where $\ca_{\alpha}(\cdot,\cdot):=
\ca(\cdot,\cdot)+\alpha(\cdot,\cdot)_{L^2(E, \mu)}$.\\

Let $K>0$ be the sector constant of $(\ca, D(\ca))$, i.e.
$$
\left |\ca_1(u,v)\right |\le K \ca_1(u,u)^{1/2}\ca_1(v,v)^{1/2}\ \ \ \mbox{for all }u,v\in D(\ca).
$$

We suppose  further that the measures $\lambda$ and $\mu$ are
equivalent. We may and will assume  that $\lambda= f_o\cdot \mu$.

Let $\mbox{cap}^\beta_{\lambda}$ be the capacity corresponding to the resolvent 
$(V_\al)_{\al>0}$, i.e. for $M\subset E ,$ it is defined by
$$
\mbox{cap}^\beta_{\lambda}(M)=
\inf\{\lambda(\overline{R}_\beta^G V_{\beta}f_o) | \,  G\in\mathcal{T},\;M\subset G \},
$$
where $\overline{R}^G_{\beta}$ is defined as $R^G_{\beta}$ but w.r.t. $\cv_{\beta}$, 
i.e.,  $\overline{R}^G_\beta u$ denotes the reduced function of $u$ on $G$ with respect to 
$\ce(\cv_\beta)$. 
Analogously to Remark \ref{rem2.1} these capacities are all equivalent for any $\beta>0$. 
Note further that for 
open sets $G$, $\mbox{cap}^\beta_{\lambda}(G)$ coincides with the 
so-called   $f_o$-capacity of $G$ associated with the 
semi-Dirichlet form $(\ca, D(\ca))$.\\

A real valued function on $E$ is called {\it $\lambda$-fine} with
respect to $\cu$ provided it is finely continuous with respect to
$\cu$ outside a $\lambda$-inessential set. Note that a real
valued function $f$ on $E$ is $\lambda$-fine with respect to $\cv$ if
and only if there exists an increasing sequence $(F_n)_n$ of
finely closed sets with respect to $\cv$, such that $\inf_n
\mbox{cap}^\beta_{\lambda}(E\setminus F_n)=0$ and $f|_{F_n}$ is finely continuous for all
$n$.

An increasing sequence of $\ct$-closed sets $(F_n)_{n \in \N}$ is called 
{\it $c_\lambda^\beta$-nest}
if
$$
\lim_{n\to \infty}c_\lambda^\beta(E\setminus E_n)=0.
$$

\vspace{2mm}

We consider the following conditions on $\cu$ and $\cv$:
\begin{enumerate}
	\item[(A1)] The sets which are $\lambda$-polar and
$\lambda$-negligible are the same for $\cu$ and $\cv$.
	\item[(A2)]  Every  $\mbox{cap}^\beta_{\lambda}$-nest is a   $c^\beta_{\lambda}$-nest.
	\item[(B1)] The function $U_\beta f$ is $\lambda$-fine with
respect to $\cv$ for every
\mbox{$f\in bp\cb$.}
	\item[(B2)] There exists a bounded strictly positive $ \cu_\beta$-excessive function  $u_o$ such that every
 $ \cu_\beta$-excessive function dominated by $u_o$ is $\lambda$-fine with respect to $\cv$.\\
\end{enumerate}

\begin{rem}   \label{conditions} % Remark 2.6
\begin{itemize}
	\item[(a)] Clearly condition $(B1)$ does not depend on $\beta >0$ and $(B2)$ for $u_o=U_\beta 1$ implies $(B1)$.
	\item[(b)] Since the resolvent $\cv$ satisfies condition $(ii.b)$ from
Proposition \ref{prop2.0} (cf. \cite{BeBo 04}), we deduce that the
following assertions are equivalent for a set $M$:
\begin{itemize}
	\item[--] the set $M$ is $\mu$-polar with respect to $\cv$ (in this case $M$ is also $\mu$-negligible);
	\item[--] the set  $M$ is $\lambda$-negligible and $\lambda$-polar with
respect to $\cv$;
\item[--] $\mbox{cap}^\beta_{\lambda}(M)=0$.
\end{itemize}
\item[(c)] By Remark \ref{rem2.1} it follows that condition $(A2)$ does not depend on $\beta >0$.
	\item[(d)] If condition $(A2)$ holds then:
\begin{enumerate}
	\item[--] A set which is  $\lambda$-polar with respect to $\cv$ is also
$\lambda$-polar with respect to $\cu$;
	\item[--] If the capacity $\mbox{cap}^\beta_{\lambda}$ on $(E, \ct)$ is tight then
$c_\lambda^\beta$ is also tight.\\
\end{enumerate}
\end{itemize}
\end{rem}
A function $g\in\cb$ is called {\it $c_\lambda^\beta$-quasi-continuous} 
(in short $c_\lambda^\beta$-q.c.)  provided  there exists a $c_\lambda^\beta$-nest 
 $(F_n)_{n \in \N}$ such that  for each $n$ $g|_{F_n}$ is continuous on $F_n$.  

\begin{thm}\label{thm2.4}  %Theorem 2.7
Assume that condition $(A2)$ is satisfied. Then the following assertions hold.
\begin{itemize}
	\item[(i)] If $(B1)$ is verified then there exists a Lusin topology on
$E$ which is natural for $\cu$ and for a trivial modification
$\cv'$ of $\cv$, such that the right processes having $\cu$ and
$\cv'$ as associated resolvents respectively are
$\mu$-standard.
	\item[(ii)] If $(B2)$ is verified then the right processes $X$ having
$\cu$ as associated resolvent is $\mu$-standard in the
original topology $\ct$. In addition, $U_{\alpha} f$ is $c_\lambda^{\beta}$-q.c. for any $f\in b\cb$ and $\alpha>0$.
\end{itemize}
\end{thm}

\begin{proof}
We have already observed that  $\cv$ (being the resolvent of a
semi-Dirichlet form)  satisfies condition $(ii.b)$ from
Proposition \ref{prop2.0}:

every $\cv_\beta$-excessive function dominated by $V_\beta f_o$ is $\lambda$-regular.\\
Consequently (see also Theorem 3.5.2 in \cite{BeBo 04}),   
the capacity $\mbox{cap}^\beta_{\lambda}$ is tight and the right process having $\cv$
as associated resolvent is $\mu$-standard in any topology
which is natural with respect to $\cv$.

$(i)$ According with the usual method of constructing  Ray cones
(see e.g. \cite{BeBo 04}) and using hypothesis $(B1)$, there
exists a Ray cone $\calr$ associated with $\cu$ such that every
$u\in\calr$ is $\lambda$-fine with respect to $\cv$. Let $\calr_o$ be
a countable subset of $\calr$ which is dense in $\calr$ in the
uniform norm. Let $M_o$ be a set which is $\mu$-polar with
respect to $\cv$, and such that every $u$ from $\calr_o$ is finely
continuous with respect to $\cv$ outside $M_o$. From \cite{BeBo
04}, page 168, it follows that there exists a  set $M$ which is
$\lambda$-inessential with respect to $\cv$ such that $M_o\subset M$.
We consider now the trivial modification $\cv'$ of $\cv$ on $M$.
Note that $\cv'$ is also the resolvent of the (quasi-regular)
semi-Dirichlet form $(\ca, D(\ca))$ on $L^2(E, \mu)$, in
particular  $\cv'$ satisfies condition $(ii.b)$ and if $A\in \cb$
then $\mbox{cap}^\beta_{\lambda}(A)=
\mbox{cap}^\beta_{\lambda}(A\setminus M)=  \lambda ('\!
R^{A\setminus M}_\beta V'_\beta f_o)=\lambda ('\! R^{A}_\beta
V'_\beta f_o)$; here $'\!R^{A}_\beta$ denotes the reduction
operator on $A$  with respect to $\cv'_\beta$. By assertion $(b)$
of Remark \ref{rem2.2} we deduce that the topology $\calr$ is
natural with respect to $\cv'$ and we conclude that the capacity
$\mbox{cap}^\beta_{\lambda}$ is tight in $\ct_\calr$. Let
$(K_n)_n\subset E$ be an increasing sequence of
$\ct_\calr$-compact sets such that
$\mbox{cap}^\beta_{\lambda}(E\setminus K_n)= \lambda ('\!
R^{E\setminus K_n}_\beta V_\beta f_o)\leq \frac{1}{2^n}$ for all
$n$. From Proposition \ref{prop2.0} we get for every $n$ a
$\ct$-open set $G_n$ such that $G_{n-1}\supset G_n \supset
E\setminus K_n$ and  $\mbox{cap}^\beta_{\lambda}(G_n)\leq
\mbox{cap}^\beta_{\lambda}(E\setminus K_n) + \frac{1}{2^n}$. By
$(A2)$ we deduce that $\inf_n c_\lambda^\beta(G_n)=0$ and
consequently  the capacity $c_\lambda^\beta$ is also tight in
$\ct_\calr$. The claimed $\mu$-standardness property follows
now by the above considerations and Proposition \ref{prop2.1}.

$(ii)$ Passing to the $\beta$-level of the resolvent $\cu$ and
according with Proposition \ref{prop2.1},  we may assume that $u_o\in
\ce(\cu)$. We consider now the "Doob $u_o$-transform of $\cu$",
namely the sub-Markovian resolvent of kernels
$\cu^o=(U^o_\alpha)_{\alpha>0}$ defined by
$$
U^o_\alpha f= \frac{1}{u_o} U_\alpha (u_o f), \quad f\in p\cb, \
\alpha>0.
$$
The following assertions hold.

-- If $v\in p\cb$ then: $v\in \ce(\cu_\beta)$ if and only if
$\frac{v}{u_o}\in \ce(\cu^o_\beta)$.

--  A $\sigma$-finite measure $\xi$ on $(E, \cb)$ is
$\cu_\beta$-excessive (i.e., $\xi\circ \alpha U_{\beta+\alpha}
\leq \xi$ for all $\alpha>0$) if and only if the measure
$\frac{1}{u_o}\cdot \xi$ is $\cu^o_\beta$-excessive. Moreover
$\xi$ is a potential with respect to $\cu_\beta$ (i.e., $\xi=\nu
\circ U_\beta$, where $\nu$ is a $\sigma$-finite measure on $E$)
if and only if $\frac{1}{u_o}\cdot \xi$ is a potential with
respect to $\cu^o_\beta$. Consequently, $\cu^o$ is the resolvent of
a Borel right process with state space $E$ (cf. Section 1.8 in
\cite{BeBo 04} and \cite{BeBoRo 06}).

-- The fine topologies  on $E$ with respect to $\cu$ and $\cu^o$
coincide. In particular, a topology on $E$ is simultaneously
natural with respect to $\cu$ and $\cu^o$. The
$\lambda$-inessential sets with respect to $\cu$ and $\cu^o$ are
the same.

-- $\cu^o$  and $\cv$ satisfy  condition $(A2)$. Condition $(B2)$ implies
that:

\noindent $(B2')\quad$ Every bounded  $ \cu^o_\beta$-excessive
function is $\lambda$-fine with respect to $\cv$.

As in the proof of Proposition \ref{prop2.0} we consider  a Ray
cone $\calr$ (formed by $\cu_\beta$-excessive functions) such that
the topology $\ct_\calr$ generated by $\calr$ is finer than the
original topology $\ct$. Similarly, there exists a Ray cone
$\calr^o$ (formed by $\cu^o_\beta$-excessive functions) such that
the topology $\ct_{\calr^o}$ generated by $\calr^o$ is finer than
$\ct_\calr$. As before,  we consider $\calr^o_o$, $M_o$, $M$ and
$\cv'$. From $(B2')$ it follows that $\ct_{\calr^o}$ is natural
with respect to $\cv'$. Like in the proof of assertion $(i)$ we
get now that the capacity $\mbox{cap}^\beta_{\lambda}$ is tight in
$\ct_{\calr^o}$ and since $\ct_{\calr} \subset \ct_{\calr^o}$ we
deduce by $(A2)$ that the capacity $c^\beta_\lambda$ is  tight in
$\ct_{\calr}$. Consequently, Proposition \ref{prop2.1} implies the
$\lambda$-standardness property for $X$ in $\ct_\calr$ and clearly
also in the original topology $\ct$.

Let now $f\in bp\cb$. It is sufficient to prove the second assertion 
of $(ii)$ for $\alpha=\beta$.
Since $U_\beta f$ is a $\lambda$-regular $\cu_\beta$-excessive function, 
we deduce by  Proposition 3.2.6 from  \cite{BeBo 04}
that it is $c_\lambda^\beta$-q.c. in $\ct_{\cal R}$. Let $(K_n)_{n\in \N}$ 
be a $c_\lambda^\beta$-nest of $\ct_{\cal R}$-compact sets, such that
$U_\beta f |_{K_n}$ is continuous for each $n$. 
Because $\ct|_{K_n}=  \ct_{\cal R}|_{K_n}$ for all $n$, we conclude that
$U_\beta f$ is  also $c_\lambda^\beta$-q.c. in $\ct$.
\end{proof}

Recall that the {\it weak duality hypothesis}
(with respect to the topology $\ct$ and a measure
$m$) is satisfied by $\cu$ and  a sub-Markovian resolvent
$\cu^*=(U^*_\alpha)_{\alpha>0}$ provided that $\cu^*$ is
also the resolvent of a right Markov process with state space $E$
and for all $f, g\in p\cb$, $\alpha >0$ we have
$$
\int f U_\alpha g \ dm = \int gU^*_\alpha f \ dm.
$$
(see e.g. \cite{GeSh 84} and \cite{BeBo 04}).

Note that $\cv$ (being the resolvent of a semi-Dirichlet form) fits in the frame of the weak 
duality, by choosing an appropriate measure $m$  which is equivalent with $\mu$; see section 7.6 in \cite{BeBo 04}.\\

We assume further that  the measure $\lambda$ is absolutely
continuous with respect to $m$.

\begin{remark}\label{weakdualityhyp}  % Remark 
If the weak duality hypothesis is satisfied by $\cu$ and $\cu^*$
(with respect to the topology $\ct$ and the measure $m$) then the
capacity $c^\beta_\lambda$ is tight in the topology $\ct$ (see
\cite{BeBo 05}). Moreover, by the result of J.B. Walsh \cite{Wa
72} the process $X$ has c\`adl\`ag trajectories.
\end{remark}

The next theorem shows that, under the weak duality hypothesis for $\cu$, 
it is $(A1)$ the adequate condition leading to the standardness property of the process $X$.

\begin{thm}\label{thm2.5} %Theorem 2.8
Suppose that the {\it weak duality hypothesis} (with respect to
the topology $\ct$ and the measure $m$) is satisfied by $\cu$ and
$\cu^*$. Assume that condition $(A1)$ is verified by  $\cu^*$ and
$\cv$, and that every point of $E$ is $\mu$-polar with respect to
$\cv$. Then the following assertions hold.
\begin{itemize}
	\item[(i)] The right process $X$ having $\cu$ as associated resolvent
is $\mu$-standard in the topology $\ct$.
	\item[(ii)] Every $\cu_\beta$-excessive function dominated by $p_o$ has
a $c^\beta_\lambda$-quasi continuous version.
\end{itemize}
\end{thm}
\begin{proof}
We show that the axiom of $\lambda$-polarity holds for $\cu^*$
(i.e., every semipolar set is $\lambda$-polar with respect to
$\cu^*$). Let $M\in \cb$ be a semipolar set. By Proposition 1.7.27
and Corollary 3.2.16 in \cite{BeBo 04} there exists a measure
$\nu$ carried by $M$ such that a subset of $M$ is $\lambda$-polar
and $\lambda$-negligible if and only if it is $\nu$-negligible;
such a measure $\nu$ is called {\it Dellacherie measure}. By
condition $(A1)$ it follows that $\nu$ is a Dellacherie measure
with respect to $\cv$ and consequently $M$ is a $\mu$-semipolar
set. Since the axiom of $\mu$-polarity holds for the
semi-Dirichlet forms (see e.g. \cite{Fi 01} and Corollary 7.5.20 from \cite{BeBo 04}) 
we
conclude that the set $M$ is $\mu$-polar and therefore (by Remark \ref{conditions} $(b)$)
it is  $\lambda$-negligible and $\lambda$-polar  with respect to $\cv$. 
Again from $(A1)$ we conclude  that $M$ is $\lambda$-polar  with respect to  $\cu^*$.

Theorem 7.2.9 in \cite{BeBo 04} implies now that assertion $(ii)$
holds. By Theorem 3.5.2 and Proposition 3.5.3 in \cite{BeBo 04},
and Proposition \ref{prop2.1} we conclude that assertion $(i)$
also holds.
\end{proof}

\begin{rem}\label{final remark}  % Remark 2.9
The results from Theorem \ref{thm2.4} and Theorem \ref{thm2.5}
remain valid under the following hypothesis (which is weaker than
assuming that $\cv$ is the resolvent of a semi-Dirichlet form):
condition $(ii.b)$ from  Proposition \ref{prop2.0} holds.
\end{rem}

\section{Applications}\label{3}  %Section 3
In this section we derive the quasi-regularity property of generalized Dirichlet forms related to the resolvent $\cu$. 
Note that we do not aim to derive the $\mu$-specialty of the associated processes. 
In particular for the quasi-regular generalized Dirichlet forms all results from \cite{St1}, and \cite{Tr1} will be available 
as far as the $\mu$-specialty is not concerned. 
However, this property is not really relevant for stochastic calculus and potential theory. It is mainly used to show the equivalence of a 
certain class of processes with a certain class of bilinear forms (cf. \cite[IV. Theorem 3.5 and Theorem 5.2]{mr} and 
\cite[IV. Theorem 3.1 and Corollary 3.2]{St1}).\\ 
Throughout this section we assume that $E$ is a {\it metrizable Lusin space}. However, this assumption is only used 
in order to apply the tightness result of Remark \ref{quasileftandcap} $(c)$ .

\subsection{On the quasi-regularity of generalized Dirichlet forms}\label{3.1}  %Subsection 3.1
If not otherwise stated we maintain the notations of section \ref{2}. 
In particular, $\mu$ is a $\sigma$-finite
Borel measure on the Lusin space $(E, \ct)$ that has full support,
$(\ca, D(\ca))$ is a semi-Dirichlet form on $L^2(E, \mu)$ that is associated with a right
process $Y$ on $E$, and $\cv=(V_\al)_{\al>0}$ is the process resolvent of $(\ca, D(\ca))$.
From \cite{Fi 01} (see also Theorem 7.6.3  from \cite{BeBo 04}) we know that then $(\ca, D(\ca))$
is automatically quasi-regular.
We assume that there is a generalized Dirichlet form $(\ce,\cf)$
on $L^2(E,\mu)$ with associated $L^2(E,\mu)$-resolvent $(G_{\alpha})_{\alpha>0}$ (
see \cite{St1}) and let $\lambda=f_o d\mu$ as in the previous section.\\

In general $(\ce,\cf)$ is written as

\[ {\ce}(u,v):= \left\{ \begin{array}{r@{\quad\quad}l}
 {\cq}(u,v)-\langle\Lambda u,v \rangle & \mbox{ for}\ u \in{\cf},\ v\in D(\cq) \\
            {\cq}(u,v)-\langle\widehat\Lambda v, u \rangle &
            \mbox{ for}\ u\in D(\cq),\ v\in\widehat
            {\cf}, \end{array} \right. \] \\
where $(\cq,D(\cq))$ is a coercive closed form on $L^2(E,\mu)$
with sector constant $K'>0$, and  $\langle\cdot, \cdot\rangle$
denotes the dualization between $D(\cq)'$ and $D(\cq)$ (for the
definition of further notions cf. again \cite{St1}). In
particular,  by \cite[I. Proposition 3.4]{St1} we have
$$
G_{\alpha}(L^2(E, \mu))\subset D(\cq) \mbox{ densely,}
$$
and
$$
\ce_{\alpha}(G_{\alpha}f,u)=(f,u)_{L^2(E,\mu)},
$$\\
for all $\alpha>0$, $f\in L^2(E,\mu)$, and $u\in D(\cq)$, where
$\ce_{\alpha}(\cdot,\cdot):=\ce(\cdot,\cdot)+\alpha(\cdot,\cdot)_{L^2(E,\mu)}$.
In case $\cq=0$ it is considered that $D(\cq)=L^2(E,\mu)$.
It further holds by \cite[I. Lemma 2.5(ii)]{St1} that
$$
\cq(u,u)\le \ce(u,u) \ \ \mbox{ for all } u\in \cf,
$$
and the $\cf$-norm is defined by
$$
\|u\|_{\cf}^2=\|u\|_{D(\cq)}^2+\|\Lambda u\|_{D(\cq)'}^2, \ \ \ u\in \cf.
$$\\
It follows that for all $u\in \cf$ and $v\in D(\cq)$
\begin{eqnarray*}
|\ce_1(u,v)| & \le & K'\|u\|_{D(\cq)}\|v\|_{D(\cq)}+\|\Lambda u\|_{D(\cq)'}\|v\|_{D(\cq)}\\
& \le &  K'\|u\|_{\cf}\|v\|_{D(\cq)}+\|u\|_{\cf}\|v\|_{D(\cq)}= (K'+1)\|u\|_{\cf}\|v\|_{D(\cq)}.
\end{eqnarray*}

Let us consider the following assumptions on $(\ce,\cf)$:\\ \\
$(C1)\quad$ There exists a constant $C>0$ such that
$$
\ca_1(u,u)\le C \ce_1(u,u)\quad  \mbox{ for all } u\in \cf.
$$
$(C2)\quad$ $G_{\gamma} (L^2(E,\mu)_b)\subset D(\ca)$ densely for some $\gamma>0$.

\begin{rem}\label{exgendf}   % Remark 3.1
\begin{itemize}
	\item[(a)] If $(C2)$ holds for some $\gamma>0$, then it holds for all
$\gamma>0$.
	\item[(b)] The assumption that $(\ce,\cf)$ is related to a 
semi-Dirichlet form $(\ca, D(\ca))$ by $(C1)$, and
$(C2)$ is quite natural. Indeed, if $(\cq,D(\cq))=(\ca, D(\ca))$
then $(C1)$, and $(C2)$ hold by definition (see \cite{St1}).
Typical examples where $(\cq,D(\cq))=(\ca, D(\ca))$ are the
time-dependent Dirichlet forms in \cite{o}, and \cite{RuTr}. If
$\cq=0$  on $D(\cq)=L^2(E,\mu)$, then $(\Lambda,\cf)$ is the $L^2(\mu)$-generator 
of some $C_0$-semigroup of contractions on $L^2(\mu)$.
In the applications (cf. e.g. \cite{St2}, \cite{Tr3}, \cite{Tr6}) 
$\langle -\Lambda u,v \rangle$ can uniquely
be extended at least for $u,v\in\cf$ to $\ca^0(u,v)-\cn(u,v)$
where $\ca^0$ is a (quasi-regular) semi-Dirichlet form and $\cn$
is some (non-sectorial) positive bilinear form  which is
represented by a $\mu$-divergence free vector field $B$. The
latter means that $\cn(u,v)=\int_E \langle B,\nabla u\rangle
\,v\,d\mu$ and $-\cn(u,u)=-\frac12\int_E \langle B,\nabla
u^2\rangle \,d\mu=0$ for enough functions $u$. Consequently,
$(C1)$ holds with $\ca=\ca^0$, and it turns out in \cite{St2},
\cite{Tr3}, \cite{Tr6}, that $(C2)$ also holds.
\end{itemize}
\end{rem}

An element $u$ of $L^2(E,\mu)$ is called 1-excessive w.r.t. $\ce$ if
$\alpha G_{\alpha+1}u\le u$ ($\mu$-a.e.) for all $\alpha \ge 0$.
Let $\cp$ denote the 1-excessive elements w.r.t. $\ce$ in $D(\cq)$.
Define $\cp_{\cf}:=\{u\in \cp \mid \exists f\in \cf,\ u\le f\}$.
It can be shown that $\cp_{\cf}=\{u \mbox{ 1-excessive} \mid \exists f\in \cf,\ u\le f\}$
(see \cite[III. Lemma 2.1(i)]{St1}).
For a $\ct$-open set $U$ and an element $u\in L^2(E,\mu)$ such that $u\cdot 1_U\le f$
for some $f\in \cf$,
let $u_U:=e_{u\cdot 1_{U}}$ be the 1-{\it reduced function} of $u\cdot 1_{U}$ as defined in
\cite[Definition\,III.1.8.]{St1}. By \cite[III. Proposition 1.7]{St1} and
\cite[III. Lemma 2.1(ii)]{St1} $\cp_{\cf} \ni u_U\le u$ and $u_U=u$ on $U$ for any $u\in \cp_{\cf}$.\\
For $u \in \cp_{\cf}$ there exists $u_U^{\alpha}\in \cf\cap \cp$ such that
$(0\le)u_U^{\alpha}\le u_U^{\beta}$, $0<\alpha\le\beta$, $u_U^{\alpha}\rightarrow u_U$, $\alpha\to \infty$,
strongly in $L^2(E,\mu)$ and weakly in $D(\cq)$ and
\begin{eqnarray}\label{reduced}
\ce_1(u_U^{\alpha},v)=\alpha((u_U^{\alpha}-u\cdot 1_U)^-,v)_{L^2(E,\mu)} \mbox{ for any } v\in D(\cq).
\end{eqnarray}
where $(u_U^{\alpha}-u\cdot 1_U)^-$ denotes the negative part of $u_U^{\alpha}-u\cdot 1_U$
(see \cite[III. Proposition 1.6 and proof of Proposition 1.7]{St1}).
Moreover, the solutions $u_U^{\alpha}$, $\alpha>0$, to (\ref{reduced})
are uniquely determined.\\ \\
An increasing sequence of closed subsets $(F_k)_{k\ge 1}$
is called an $\ce$-nest, if for every function $u\in \cp\cap \cf$ it
follows that $u_{E\setminus F_k}  \to  0$ in $L^2(E,\mu)$ and weakly
in $D(\cq)$. Since with $\varphi:=f_o$ in  \cite[III. Proposition 2.10]{St1}) we have 
$\mbox{Cap}_{f_o}(E\setminus F_k)=\int_E (G_1 f_o )_{E\setminus F_k}f_o d \mu$,  it follows by the same proposition that 
$(F_k)_{k\ge 1}$ is an $\ce$-nest if and only if  
\begin{eqnarray}\label{capgdf}
\lim_{k\to \infty}\mbox{Cap}_{f_o}(E\setminus F_k)= \lim_{k\to \infty} \int_E (G_1 f_o )_{E\setminus F_k}f_o d \mu=0.
\end{eqnarray}

A subset $N\subset E$ is called  {\it $\ce$-exceptional}  if there is an
$\ce$-nest $(F_k)_{k\ge 1}$ with ${\displaystyle N\subset \bigcap_{k\ge 1}
  (E\setminus F_k)}$. 
  
A property of points in $E$ holds {\it $\ce$-quasi-everywhere}  (abbreviated $\ce$-q.e.)
if the property holds outside some $\ce$-exceptional set. 

A function $f$ defined up to some $\ce$-exceptional set $N\subset E$
is called {\it $\ce$-quasi-continuous}  ($\ce$-q.c.) if there exists an
$\ce$-nest $(F_k)_{k\ge 1}$, such that $\bigcup_{k\ge 1}
F_k\subset E \setminus N$ and $f_{\mid F_k}$ is continuous for all $k$.

For later purposes we state the following definitions (see \cite[IV. Definitions 1.4 and 1.7]{St1}):
\begin{defn}
$X$ is said to be {\rm properly associated} in the resolvent sense with $(\ce,\cf)$, if $U_\alpha g$ is 
$\ce$-q.c. for any $g\in L^2(E;\mu)\cap b\cb$, and any $\alpha>0$.
\end{defn}

\begin{defn}
The generalized Dirichlet form $(\ce,\cf)$ on $L^2(E,\mu)$
is called {\rm quasi-regular} if:
\begin{itemize}
\item[$(i)$] There exists an $\ce$-nest $(E_k)_{k\ge 1}$ consisting of compact sets.
\item[$(ii)$] There exists a dense subset of $\cf$ whose elements have $\ce$-q.c. $\mu$-versions.
\item[$(iii)$]  There exist $u_n\in\cf$, $n\in\N$, having $\ce$-q.c. $\mu$-versions $\widetilde u_n$, $n\in\N$, 
and an $\ce$-exceptional set $N\subset E$ such that $\{\widetilde u_n\mid n\in\N\}$ separates 
the points of $E\setminus N$.
\end{itemize}
\end{defn}

\begin{lemma}\label{nest} % Lemma 3.2
Let $(C1)$, $(C2)$ hold. Then we have:
\begin{itemize}
    \item[$(i)$] $\cp_{\cf}\subset D(\ca)$, i.e. every 1-excessive function w.r.t.
    $\ce$ that is dominated by an element of $\cf$ is in $D(\ca)$.
    \item[$(ii)$] $u_U\in D(\ca)$, and $\|u_U\|_{D(\ca)}\le (\sqrt{C}(K'+1)+1)\|f\|_{\cf}$
    for any $u\in \cp_{\cf}$ with $u\le f\in \cf$ and any $\ct$-open set $U$.
    \item[$(iii)$] Every $\ce$-nest is an $\ca$-nest.
\end{itemize}
\end{lemma}

\begin{proof}
$(i)$ Since $u_E=u$ for any $u\in \cp_{\cf}$ (i) follows in particular from (ii) if we put $U=E$. \\
$(ii)$ Let $U\in \ct$, $u\in \cp_{\cf}$, $u\le f\in \cf$, and
$u_U^{\alpha}$ for $\alpha>0$ be the unique solutions of
(\ref{reduced}). It follows by $(C1)$ that
\begin{eqnarray*}
\ca_1(u_U^{\alpha}-f,u_U^{\alpha}-f)\le C\ce_1(u_U^{\alpha}-f,u_U^{\alpha}-f).
\end{eqnarray*}
Since $\ce_1(u_U^{\alpha},u_U^{\alpha}-f)\le 0$
by \cite[III. Proposition 1.4.(ii)]{St1} we get
\begin{eqnarray*}
\ca_1(u_U^{\alpha}-f,u_U^{\alpha}-f)\le C(K'+1)\|f\|_{\cf}\|u_U^{\alpha}-f\|_{D(\cq)}.
\end{eqnarray*}
for any $\alpha>0$. This also holds for $\ca$ replaced with $\cq$ and $C=1$,
thus  $\|u_U^{\alpha}-f\|_{D(\cq)}\le (K'+1)\|f\|_{\cf}$, and so
$$
\|u_U^{\alpha}-f\|_{D(\ca)}\le \sqrt{C}(K'+1)\|f\|_{\cf}.
$$
Since $u_U^{\alpha}\to u_U$ in $L^2(E,\mu)$ as $\alpha\to \infty$
it follows from \cite[I. Lemma 2.12]{mr} that $u_U\in D(\ca)$ and
$$
\ca_1(u_U-f,u_U-f)\le \liminf_{\alpha\to \infty}\ca_1(u_U^{\alpha}-f,u_U^{\alpha}-f).
$$
Consequently $\|u_U-f\|_{D(\ca)}\le \sqrt{C}(K'+1)\|f\|_{\cf}$. It follows
$$
\|u_U\|_{D(\ca)}\le \|u_U-f\|_{D(\ca)}+\|f\|_{\cf}\le (\sqrt{C}(K'+1)+1)\|f\|_{\cf}
$$
as desired.\\
$(iii)$  Let $(F_k)_{k\ge 1}$ be an $\ce$-nest. By definition of
$\ce$-nest we have that $g_{E\setminus F_k}\to 0$ in $L^2(E,\mu)$ as $k \to
\infty$ for any $1$-excessive function $g \in \cf$. By (ii) we
obtain that
$$
\sup_{k\ge 1}\|g_{E\setminus F_k}\|_{D(\ca)}\le (\sqrt{C}(K'+1)+1)\|g\|_{\cf}<\infty.
$$
Hence we
conclude that $g_{E\setminus F_k}\to 0$ weakly in $D(\ca)$ as
$k \to \infty$ for any $1$-excessive function $g\in \cf$.
Now let $f\in L^2(E,\mu)_b$ be arbitrary.
By the preliminary considerations we know that
$$
f_k:=G_1 f-(G_1 f^+)_{E\setminus F_k}+(G_1f^-)_{E\setminus F_k}\in D(\ca),
$$
and that $f_k\to G_1 f$ weakly in $D(\ca)$ as $k \to \infty$
($f^+$, $f^-$ denote respectively the positive and negative parts
of $f$). By Banach-Saks theorem we know that there is a
subsequence $(n_k)_{k\ge 1}$ such that the Cesaro means
$v_N:=\frac{1}{N}\sum_{k=1}^{N}f_{n_k}$ converge strongly to $G_1
f$ in $D(\ca)$. Note that for each $N$ there is some $k$ (e.g.
$k=n_N$) such that $v_N\in D(\ca)_{F_k}:=\{u\in D(\ca)\,|\, u=0 \
\mu\mbox{-a.e. on } E\setminus F_k \}$. Using $(C2)$, and Remark
\ref{exgendf}$(a)$ we then see that
$$
\bigcup_{k\ge 1}D(\ca)_{F_k}\subset D(\ca)  \ \mbox{densely}.
$$
Thus $(F_k)_{k\ge 1}$ is an $\ca$-nest by \cite[Definition 2.9(i)]{MaOvRo}.
\end{proof}
\bigskip

Let (as in section \ref{2}) $X =(\Omega, \cf, \cf_t, X_t , \theta_t, P^{x} )$
be a Borel right Markov process
whose state space is $(E, \ct)$ and whose resolvent is $\cu=(U_\alpha)_{\alpha>0}$.\\

Now, we assume that $U_{\alpha} f$ is a $\mu$-version of
$G_{\alpha} f$ for any $\alpha>0$ and $f\in L^2(E,\mu)$, i.e., we
assume that the right process $X$ is associated with the
generalized Dirichlet form $(\ce,\cf)$.\\

As already mentioned in Remark \ref{exgendf}$(b)$ in \cite{St2},
\cite{Tr3}, \cite{Tr6}, assumptions $(C1)$, $(C2)$ hold with $C=1$
and equality in $(C1)$. This suggests that the capacities
corresponding to $\ce$ and $\ca$ should be equivalent and this is
indeed the case in \cite{St2}, \cite{Tr3}, \cite{Tr6}. A general
statement, however, even with additional assumptions is yet
unshown. Therefore we assume $(A2)$ in the following theorem and
repeat that the conditions $(A2)$, $(C1)$, and $(C2)$ are
satisfied by the "elliptic" generalized Dirichlet forms given in
\cite{St2}, \cite{Tr3}, \cite{Tr6}, and moreover  by any
semi-Dirichlet form. Note however, that  even though $(C1)$, and
$(C2)$ are trivially satisfied for the time-dependent Dirichlet
forms in \cite{o}, and \cite{RuTr},
$(A2)$ is not.\\

\begin{thm}\label{gdfstandard}  % Theorem 3.3
Suppose that $(A2)$, $(C1)$, and $(C2)$ hold. Then:

\begin{itemize}
	\item[(i)] $(B2)$ holds
for any $\beta>0$. In particular,  the right process $X$
associated to the generalized Dirichlet form $(\ce,\cf)$ is
$\mu$-tight $\mu$-standard in the original topology $\ct$. 
In addition, $X$ is properly associated in the resolvent sense with $(\ce,\cf)$.
\item[(ii)] The generalized Dirichlet form $(\ce,\cf)$ 
is quasi-regular. 
\end{itemize}
 
\end{thm}

\begin{proof}
$(i)$  It is enough to show the statement for $\beta=1$. Let $\varphi\in
L^1(E,\mu)$, $0<\varphi\le 1$. Define $u_o:=U_1 \varphi$. Let $u$
be an $ \cu_1$-excessive function such that $u\le u_o$. Since
$U_1\varphi$  is a $\mu$-version of some element in $\cf$, $u$ is
a $\mu$-version of some element in $\cp_{\cf}$. Thus by Lemma
\ref{nest}(i) $u$ is a $\mu$-version of some element in $D(\ca)$.
Since the semi-Dirichlet form $(\ca,D(\ca))$ is quasi-regular, $u$
has an $\ca$-q.c. $\mu$-version $\widetilde u$. In particular by
$(A2)$ $\widetilde u$ is $\ce$-q.c, thus $\lambda$-fine w.r.t.
$\cu$. Since $u$ is also $\lambda$-fine w.r.t. $\cu$ and
$\widetilde u=u$ $\mu$-a.e, we have that $\widetilde u=u$
$\ce$-q.e. But then by Lemma \ref{nest} $(iii)$ $\widetilde u=u$
$\ca$-q.e. and therefore $u$ is $\ca$-q.c. It follows that $u$ is
$\mu$-fine with respect to $\cv$. Consequently $(B2)$ holds as
desired.

The second assertion follows  by Theorem \ref{thm2.4} $(ii)$ and
Remark \ref{quasileftandcap} $(c)$.

$(ii)$  By  $(i)$ we have $\mu$-tightness, and so there exists an increasing
sequence $(K_n)_n$ of $\mathcal{T}$-compact sets such that
$$
\lim_{n\to \infty} \int_E (R_1^{E\setminus K_n} U_1 f_o )f_o d \mu=0.
$$
Since $R_1^{E\setminus K_n}U_1 f_o\le U_1 f_o$ and $R_1^{E\setminus K_n}U_1 f_o\in \ce(\cu_1)$ we know that 
$R_1^{E\setminus K_n}U_1 f_o$ is a $\mu$-version of some element in  $\cp_{\cf}$. Since moreover 
$R_1^{E\setminus K_n}U_1 f_o =U_1 f_o$ on $E\setminus K_n$ we 
conclude by \cite[III. Proposition 1.7 (ii)]{St1} that 
$(G_1 f_o )_{E\setminus E_k}\le R_1^{E\setminus K_n} U_1 f_o$ $\mu$-a.e. 
Hence by (\ref{capgdf})
$$
\lim_{n\to \infty} \mbox{Cap}_{f_o}(E\setminus K_n)=0,
$$
and so there exists an $\ce$-nest of $\mathcal{T}$-compact sets.
As in \cite[I. Remark 3.5]{St1} we conclude that $G_1(L^2(E;\mu)_b)\subset \cf$ densely. 
By part (i), $X$ is properly associated in the resolvent sense with $(\ce,\cf)$, i.e. $U_1 g$ is $\ce$-q.c. for any $g\in L^2(E;\mu)\cap b\cb$. 
Thus every element of the dense subset $G_1(L^2(E;\mu)_b)$ in $\cf$ admits an $\ce$-q.c. $\mu$-version.\\
Since again $X$ is properly associated 
in the resolvent sense with $(\ce,\cf)$ we obtain as in \cite[IV. Lemma 3.9]{St1} that 
$R_1^{U}U_1 f_o$ is $\ce$-q.l.s.c. Let $(K_n)_n$ be a sequence of $\mathcal{T}$-compact sets as at the beginning of this proof. 
As in \cite[IV. Lemma 3.10]{St1} we show that $P^x(\lim_{n\to\infty}D_{{E\setminus K_n}}<\zeta )=0$ for $\ce$-q.e. 
$x\in E$. The countable family of $\ce$-q.c. elements of $\cf$ that separates the points of $E$ up to an $\ce$-exceptional 
set can then e.g. be constructed as in 
\cite[IV. Lemma 3.11 and paragraph below]{St1}.\\
\end{proof}

\begin{rem}\label{gdfqr} % Remark 3.4
In \cite{mr}, and \cite{St1} the $\mu$-special property of the associated process is used in order to show that the process resolvent is quasi-continuous.
Note that we did not  use any $\mu$-special property in order to show that $X$ is properly associated in the resolvent sense with $(\ce,\cf)$. 
This is because no $\mu$-special property is used in the proof of Theorem \ref{gdfstandard}(ii). \\
\end{rem}

%*************************************************************************************************************************************
%*************************************************************************************************************************************

\subsection{Perturbation with kernels}\label{3.2}  %Subsection 3.2
In this subsection again, if not otherwise stated we maintain the notations of section \ref{2}. 
Thus $\cv=(V_\al)_{\al>0}$ is the process resolvent of a quasi-regular semi-Dirichlet form $(\ca, D(\ca))$ on $L^2(E, \mu)$, and 
$\overline{R}^M_{\beta}$ is the reduction operator on $M$ w.r.t. $\cv_\beta=(V_{\beta+\al})_{\al>0}$.

Let $P$ be a  kernel on $(E, \cb)$ such that:
\begin{itemize}
	\item[(p.1)] $P f\in \ce (\cv)$ for all $f\in p\cb$,
	\item[(p.2)] $1-P 1\in \ce (\cv)$.
\end{itemize}

For $\alpha>0$ define the kernel $P_{\alpha}$ by 
$$
P_{\alpha}f:=P f-\alpha V_{\alpha} P f, \ \ \ f\in bp\cb,
$$ 
and 
$$
U_{\alpha}:=\left (\sum_{n=0}^{\infty} P^n_{\alpha}\right )\circ V_{\alpha}.
$$

Let $Q_\alpha :=\sum_{n=1}^{\infty} P^n_\alpha$ be the associated  $\alpha$-level  "potential kernel". 
Assume that for some $\beta >0$ the kernel  $Q_\beta$ is bounded. 
Then the following assertions hold (see Proposition 5.2.4 and 5.2.5. in \cite{BeBo 04}):

\begin{itemize}
	\item[(i)] The family $\cu=(U_{\alpha})_{\alpha>0}$ is the resolvent of a right 
	process with state space $E$.
	\item[(ii)] If $M\in \cb$ and   $R^M_\beta$ (resp.  $\overline{R}^M_\beta $) denotes the kernel on $E$ induced by the 
	reduction operator on $M$ w.r.t.  $\cu_\beta$ (resp. w.r.t.  $\cv_\beta$),  then
\begin{eqnarray}\label{reduction}
R^M_\beta = \overline{R}^M_\beta + Q_\beta \overline{R}^M_\beta - R^M_\beta Q_\beta \overline{R}^M_\beta.
\end{eqnarray}
\end{itemize}

\noindent
{\bf Examples.} A first typical example of perturbing by some kernel  (see \cite{BeBo 09}) is 
produced  by the $\beta$-potential kernel $V^\beta_A$ of a continuous additive functional 
$A=(A_t)_{t\geq 0}$ of the process $Y$ associated with $\cv$,
$$
Q_\beta f(x) = V_A^\beta f(x):=E^x \int_{0}^{\infty}\!\!\! e^{-\beta t}f(Y_t)\, dA_t, \,\,
f\in p\cb, x\in E.
$$
A second example is emphasized  by  a potential theoretical approach for the 
measure-valued discrete branching (Markov) processes; see \cite{BeOp11}  for details.

\begin{rem}\label{subordinate} % Remark 3.5
\begin{enumerate}
\item[(a)] We have $V_{\alpha} \le U_{\alpha}$ for all $\alpha>0$ 
	(i.e. $\cv$ is subordinate to $\cu$; see, e.g., \cite{BeBo 04}) and $\ce (\cu_\beta)\subset 
	\ce (\cv_\beta)$ for any $\beta>0$.
\item[(b)] The fine topologies of $\cv$ and $\cu$ coincide. 
This clearly implies that conditions $(B1)$  and $(B2)$  are satisfied.
\item[(c)] Assume that instead of $\cv$ we start with the $\beta$-level resolvent $\cv_\beta$ 
(which is also the process resolvent of  a quasi-regular  semi-Dirichlet form) for some $\beta>0$. 
Then the kernel  $P_\beta$ satisfies conditions \mbox{\rm (p.1)} and \mbox{\rm (p.2)} w.r.t. $\cv_\beta$ and by the resolvent equation we get
$(P_\beta)_\alpha:=P_\beta   -\alpha V_{\beta+\alpha} P_\beta = P_{\beta +\alpha}$ for all $\alpha>0$. 
Consequently, the corresponding $\alpha$-level potential kernel is $Q_{\beta+\alpha}$ which is bounded  because 
$Q_{\beta+\alpha}\leq Q_{\beta}$ and the induced perturbed resolvent is $\cu_\beta$.
\item[(d)]  Recall  that a function $v\in\ce(\cu_\beta)$ is called {\rm universally quasi bounded} in $\ce(\cu_\beta)$  provided that for every strictly positive function 
$u\in\ce(\cu_\beta)$,  there exists a sequence $(v_i)_{i\in \N}  \subset \ce(\cv_\beta)$ such that $v=\sum_{i\in \N}  v_i$ and $v_i\leq u$ for all $i\in \N$.
We denote by $\mbox{Qbd} (\cu_\beta)$ the set of all quasi bounded elements from $\ce(\cu_\beta)$.  
The following assertions hold (see \cite{BeBo 04} for details):
\begin{itemize}
	\item[--] If $v\in \mbox{Qbd} (\cu_\beta)$ and $w\in\ce(\cu_\beta)$ with  $w\leq v$, then $w\in \mbox{Qbd} (\cu_\beta)$.
	\item[--] Every regular $\cu_\beta$-excessive functions is universally quasi bounded in $\ce(\cu_\beta)$.
	\item[--] A function $v\in\ce(\cu_\beta)$ belongs to  $\mbox{Qbd} (\cu_\beta)$ if and only if exists a sequence $(v_i)_{i\in \N}  \subset \ce(\cu_\beta)$ 
such that $v=\sum_{i\in \N}  v_i$ and $v_i \leq U_\beta f_o$ for all $i\in \N$.
	\item[--] If $\beta< \beta'$ then   $\mbox{Qbd} (\cu_\beta)\subset \mbox{Qbd} (\cu_{\beta'})$.
\end{itemize}
\end{enumerate}
\end{rem}

Let $\nu$ be any 
finite measure on $(E, \cb)$ which is equivalent with $\mu$.
Note that  in this subsection neither  $\nu$ nor $\lambda$ is assumed  to have  the density  $f_o$ with respect to $\mu$.  
Let
$$
c^\beta_{\nu}(M):=\inf\{\nu(R_\beta^G p_o) | \,  G\in\mathcal{T},\;M\subset G \}.
$$
as before, but 
$$
\mbox{cap}^\beta_{\nu,p_o}(M):=\inf\{\nu(\overline{R}_\beta^G p_o) | \,  G\in\mathcal{T},\;M\subset G \},
$$
where $p_o=U_\beta f_o$ is as in section \ref{2}. 
Note that the second set function  $\mbox{cap}^\beta_{\nu,p_o}$ is defined 
w.r.t. $p_o$ and not w.r.t.  $V_\beta f_o$ as in section \ref{2}. 
We therefore make the following remark:

\begin{rem}\label{capwithsamefunction} % Remark 3.6
Let $u\in b\ce(\cu_\beta)$. 
The functional $M\longmapsto c^\beta_{\nu, u} (M) ,$ $M\subset E ,$ defined by
$$
c^\beta_{\nu, u} (M)=\inf\{\nu(R_\beta^G u) | \,  G\in\mathcal{T},\;M\subset G \}
$$
is a Choquet capacity on  $(E, \mathcal{T})$. Clearly, if $u=p_o$ then $c^\beta_{\nu, u}= c^\beta_{\nu}$.

Then the following assertions hold.
\begin{itemize}
\item[(a)]  Assume that $u$ is a strictly positive bounded function and $u\in \mbox{Qbd} (\cu_\beta)$.  
(Such a function $u$ always exists, e.g., take $u=U_\beta 1$ and use assertion (d) of Remark \ref{subordinate}.)
Then the  conclusion of Proposition \ref{prop2.0}  
holds if we replace $c^\beta_{\nu}$ with  $c^\beta_{\nu, u}$ and condition $(ii.a)$
with the following one:\\ \\
$(ii.a')\quad$  The topology $\ct$ is a Ray one, $\ct=\ct_{\mathcal R}$, and $u\in {\mathcal R}$.\\ \\
The  assertion follows since by Proposition 1.6.3 in \cite{BeBo 04}, if $(ii.a')$ holds and $A\in \cb$, then
$c^\beta_{\nu, u}(A)=\nu(R^A_\beta u)$. Note that Theorem 3.5.2 from \cite{BeBo 04}
may be applied for the capacity $c^\beta_{\nu, u}$ because condition $(ii.b)$
is equivalent with:\\ \\
$(ii.b')\quad$  every $\cu_\beta$-excessive function dominated by $u$ is $\nu$-regular.\\ \\
The equivalence between $(ii.b)$ and $(ii.b')$ is a consequence of the following facts 
(for details see sections 2.4, 3.1 and 3.2 from \cite{BeBo 04}):
Since $u\in \mbox{Qbd} (\cu_\beta)$  and  $U_\beta f_o >0$,
 there exists a  sequence $(u_i)_{i\in \N}$ in  
$b\ce(\cu_\beta)$ such that $u=\sum_{i\in \N} u_i$ and $u_i\leq U_\beta f_o$ for all $i\in \N$. 
The Riesz decomposition property from the
cone of potentials $b\ce(\cu_\beta)$ is also used: if $v \in b\ce(\cu_\beta)$ and
$v\leq \sum_{i\in \N}  u_i$,  then there exists a sequence  
$(v_i)_{i\in \N} \subset b\ce(\cu_\beta)$ such that 
$v=\sum_{i\in \N}  v_i$ and $v_i\leq u_i$ for all $i\in \N$.
\item[(b)]
%Define the capacity  $\mbox{cap}^\beta_{\nu, u}$  analogously with $c^\beta_{\nu, u}$. 
Suppose that $p_o=U_\beta f_o$ belongs to $\mbox{Qbd} (\cv_\beta)$.  
Then the  assertions from Theorem \ref{thm2.4} hold 
if we assume that condition (A2) is satisfied by  $\mbox{cap}^\beta_{\nu, p_o}$ and $c^\beta_\nu$. 
Indeed, using the above  assertion $(a)$ for the resolvent $\cv$ instead of $\cu$, 
we can apply Proposition \ref{prop2.0} for $\mbox{cap}^\beta_{\nu, p_o}$, 
considering the Ray cone ${\mathcal R}$ such that  $p_o \in {\mathcal R}$.
\end{itemize}
\end{rem}

By (\ref{reduction}) 
it follows for open $G$ that
\begin{eqnarray}\label{capacities}
\mbox{cap}^\beta_{\nu,p_o}(G)\le c^\beta_{\nu}(G)\le \mbox{cap}^\beta_{\nu+\nu\circ Q_\beta,p_o}(G).
\end{eqnarray}

The next result shows that condition $(A2)$ holds in the case of perturbation with kernels, 
allowing  a second application of  Theorem \ref{thm2.4}.

\begin{prop}\label{capacitiesperturbation} % Proposition 3.7 
\begin{itemize}
	\item[(i)] Let $\lambda$ be the finite measure defined on $E$ by $\lambda:=\nu+\nu\circ Q_\beta$. 
Then condition $(A2)$ is satisfied by $\mbox{cap}^\beta_{\lambda, p_o}$ and $c^\beta_\lambda$. 
More precisely, if $(G_n)_{n\in \N}\subset \ct$ is decreasing then:
\begin{eqnarray*}
\inf_{n\in \N}\mbox{cap}^\beta_{\lambda, p_o}(G_n)=0\ \  \Longleftrightarrow \ \ \inf_{n\in \N}c^\beta_{\lambda}(G_n)=0.
\end{eqnarray*}
\item[(ii)] Assume  that $Q_\beta 1$ belongs to $\mbox{Qbd} (\cv_\beta)$ and  that   $\mu\circ P \ll \mu$.
Then the right process having $\cu$ as associated resolvent is $\mu$-standard 
in the original topology $\ct$.
\item[(iii)] Assume  that the measure $\mu \circ Q_{\beta}$ charges no  $\ca$-exceptional set. 
Then condition $(A2)$ is satisfied by $\mbox{cap}^\beta_{\lambda, p_o}$ and $c^\beta_\lambda$ for any finite measure $\lambda$ equivalent with $\mu$.
If in addition $Q_\beta 1 \in \mbox{Qbd} (\cv_\beta)$,  then the right process having $\cu$ as associated resolvent is $\mu$-standard  in the original topology $\ct$.
\end{itemize}
\end{prop}

\begin{proof}
$(i)$ The implication \rq\rq$\Longleftarrow \lq\lq$ is clear since 
$\mbox{cap}^\beta_{\lambda, p_o}\le c^\beta_{\lambda}$. We show now that 
\begin{eqnarray}\label{absolutelycontinuous}
\lambda\circ Q_\beta \ll \lambda. 
\end{eqnarray}
Indeed, if $f\ge 0$ and $\lambda(f)=0$ then $\nu(f)=0$ and $\nu(Q_\beta f)=0$. 
Since $P_\beta Q_\beta\le Q_\beta$ we  get $0\le P^n_\beta Q_\beta f\le Q_\beta f=0$ 
$\nu$-a.e. for any $n$, hence $Q_\beta(Q_\beta f)=
\sum_{n=1}^{\infty} P^n_\beta Q_\beta f=0$ $\nu$-a.e, i.e. 
$\nu\circ Q_\beta(Q_\beta f)=0$. Hence $\lambda(Q_\beta f)=\nu(Q_\beta f)+\nu\circ Q_\beta(Q_\beta f)=0$.

Note that 
\begin{eqnarray}\label{capacities2}
\inf_{n\in \N}\mbox{cap}^\beta_{\lambda, p_o}(G_n)=0\ \  
\Longleftrightarrow \ \ \inf_{n\in \N} \overline{R}_\beta^{G_n} p_o=0\ \ \lambda\mbox{-a.e.}
\end{eqnarray}

Let $(G_n)_{n\in \N}\subset \ct$ be such that $\inf_{n\in \N} \mbox{cap}^\beta_{\lambda, p_o}(G_n)=0$. 
Then by (\ref{capacities2}) 
we get that $\inf_{n\in \N} \overline{R}_\beta^{G_n} p_o=0$ $\lambda$-a.e. 
By (\ref{absolutelycontinuous}) the last equality holds 
$(\lambda+\lambda\circ Q_\beta)$-a.e, hence again by (\ref{capacities2}) we get 
$$
\inf_{n\in \N}\mbox{cap}^\beta_{\lambda+\lambda\circ Q_{\beta}, p_o}(G_n)=0.
$$
From (\ref{capacities}) applied to $\lambda$ we conclude that 
$\inf_{n\in \N}c^\beta_{\lambda}(G_n)\le \inf_{n\in \N} \mbox{cap}^\beta_{\lambda+\lambda\circ Q_\beta, p_o}(G_n)=0$.

$(ii)$  Let $\lambda$ be as in $(i)$. Observe that from the  assumption  $\mu\circ P \ll \mu$ 
we deduce  that the measures $\lambda$ and $\mu$ are equivalent. 
Since $U_\beta =V_\beta + Q_\beta V_\beta$, we have 
$p_o= V_\beta f_o+ Q_\beta V_\beta f_o$. As a consequence, using the hypothesis on $Q_\beta 1$ and Remark \ref{subordinate} $(d)$, 
it turns out that $p_o\in \mbox{Qbd} (\cv_\beta)$. Remark \ref{subordinate} $(b)$
implies that $(B2)$ holds, while by assertion $(i)$ it follows that 
$(A2)$ is  satisfied by $\mbox{cap}^\beta_{\lambda, p_o}$ and $c^\beta_\lambda$.
From Remark \ref{capwithsamefunction} $(b)$ and  
Theorem \ref{thm2.4} $(ii)$ we conclude now that assertion $(ii)$ holds.

$(iii)$ Recall that by definition $\lambda\circ Q_{\beta}$ charges no  $\ca$-exceptional set  if $\lambda\circ Q_{\beta}(M)=0$ for any 
$M\subset \bigcap_{k\ge 1}(E\setminus F_k)$, where $(F_k)_{k\ge 1}$ is an $\ca$-nest. 
Since $\ca$ is quasi-regular,  this is equivalent to saying that $\lambda\circ Q_{\beta}$ charges no $\nu$-polar sets w.r.t. $\cv$.
%As in $(i)$ the implication \rq\rq$\Longleftarrow \lq\lq$ is clear since $\mbox{cap}^\beta_{\lambda, p_o}\le c^\beta_{\lambda}$.\\

Let $(E_n)_{n\ge 1}$ be a  $\mbox{cap}^\beta_{\lambda, p_o}$-nest, i.e.,  
$ \mbox{cap}^\beta_{\lambda, p_o}(G_n)\searrow 0$ as $n\to\infty$,  with 
$G_n:=E\setminus E_n$. 
By the  quasi-regularity of $\ca$ and monotonicity of $\overline{R}^{G_n}_\beta U_{\beta}f_o$,  
there exists a second   nest $(\overline{E}_k)_{k\ge 1}$ such that pointwise 
$\overline{R}^{G_n}_\beta U_{\beta}f_o\searrow 0$ as $n\to \infty$ on each $\overline{E}_k$. 
Since by hypothesis the finite measure  $\lambda \circ Q_{\beta}$ charges no  $\ca$-exceptional set,  we get
that  $(\overline{R}^{G_n}_\beta U_{\beta}f_o)_n$  is a sequence of bounded functions decreasing to zero
$\lambda\circ Q_{\beta}$-a.e. Consequently, we have 
$\lambda\circ Q_{\beta}(\overline{R}^{G_n}_\beta U_{\beta}f_o) \searrow 0$ as $n\to \infty$.
By (\ref{reduction})
\begin{eqnarray*}
R^{G_n}_\beta U_{\beta}f_o  & \le  & \overline{R}^{G_n}_\beta U_{\beta}f_o + Q_\beta \overline{R}^{G_n}_\beta U_{\beta}f_o\, , \\
\end{eqnarray*}
hence 
\begin{eqnarray*}
c_{\lambda}^{\beta}(G_n) & \le  & \mbox{cap}^\beta_{\lambda, p_o}(G_n)+\int_E Q_\beta \overline{R}^{G_n}_\beta U_{\beta}f_o(x) \lambda(dx).\\
\end{eqnarray*}
It follows  that $c_{\lambda}^{\beta}(G_n)\searrow 0$ as $n\to \infty$, i.e.,  $(E_n)_{n\ge 1}$ is a  $c^{\beta}_\lambda$-nest.\\

The proof of the last assertion of $(iii)$ is similar to that of $(ii)$.
\end{proof}

\subsection{Quasi-regularity of  generalized Dirichlet forms obtained by perturbation with kernels}\label{3.2.1}     %Subsection 3.3
In this subsection we want to show that there exists a quasi-regular generalized Dirichlet form  that is associated to the perturbation of the
quasi-regular semi-Dirichlet form $(\ca, D(\ca))$ on $L^2(E, \mu)$ with kernels. Our results are in particular related to 
\cite[Remark 3.3.(iv)]{RoTr}. \\ 
Let $\beta>0$ be as in subsection \ref{3.2}.
\begin{defn}
We say that $\cv$ satisfies the {\bf absolute continuity condition} if 
$$
V_{\beta}(x,\cdot) \ll \mu \mbox{ for all } x\in E, 
$$
i.e. $V_{\beta}(x,\cdot)$ is absolutely continuous w.r.t. $\mu$ for each $x$.
\end{defn}
\smallskip

\begin{remark}
In finite dimensions the absolute continuity condition is typically guaranteed through embedding 
theorems of (weighted) Sobolev spaces in H\"older spaces. For instance, 
it is satisfied if $H_{\alpha} f$ (the $L^p(E,\mu)$-version of $V_{\alpha}f$), admits a  
H\"older continuous $\mu$-version for any $\alpha>0$ and $f\in L^p(E;\mu)$. 
\end{remark}
\bigskip
For the rest of the section we assume that $E$ is a {\it separable (and metrizable) space}. \\ \\
Let $\{x_k;k\ge 1\}\subset E$ be any dense subset and define the finite measure $\nu$ w.r.t. $\{x_k;k\ge 1\}$ as 
$$
\nu :=\sum_{k\ge 1} 2^{-k}\delta_{x_k}.
$$
Consider now the potential $\cv_\beta$-excessive measure  $\xi$ (resp. the $\cu_\beta$-excessive measure $\eta$) defined as
$$
\xi:=\nu\circ V_\beta \quad (\mbox{resp. }  \eta:= \nu\circ U_\beta).
$$
Let further 
$$
{p}_t f(x):=p_t(x,f):={E}^x[f({X}_t)],\ x\in E, \ t\ge 0, \ f\in p\cb,
$$\\
denote the transition semigroup of the right process $(X_t)_{t\ge 0}$ with state space  $E$,
corresponding to $\cu=(U_{\alpha})_{\alpha>0}$. 

The right process $(X_t)_{t\ge 0}$ is said to be {\it transient}, if  
\begin{itemize}
\item[{\bf (T)}] there exists  $\varphi>0$, universally Borel measurable and $U_0\varphi(x)=E^x[\int_0^{\infty}\varphi(X_t)dt]<\infty$ 
for all $x\in E$. \ \ 
\end{itemize}
For more details about this  hypothesis see, e.g.,  \cite{Ge 80}.

If ${\bf (T)}$  holds we  define the $\cu$-excessive measure $\eta_0$ on $E$  as
$$
 \eta_0:= \nu_0\circ U_0, \ \ \ \mbox{where }\ \ \ \nu_0 :=\sum_{k\ge 1} \frac{2^{-k}}{U_0\varphi(x_k)}\delta_{x_k}.
$$

\begin{rem} \label{abscont} %Remark 3.8
\begin{itemize}
	\item[(a)] Clearly, $\xi$ and $\eta$ are finite measures while $\eta_0$  is $\sigma$-finite,  provided that ${\bf (T)}$  holds.
In addition, all these measures have full support.

	\item[(b)] Assume that ${\bf (T)}$  holds. If  $W$ is a kernel on $E$, then because $\nu$ and $\nu_0$ are equivalent measures, 
it follows that the measures  $\nu\circ W$ and $\nu_0 \circ W$ are also equivalent. It is also clear that  if $\kappa$ is a $\sigma$-finite 
measure then $\kappa\circ U_0$ and $\kappa\circ U_{\alpha}$ are equivalent for any $\alpha>0$. 
In particular, all the measures
$\eta=\nu\circ U_\beta$, $\nu_0 \circ U_\beta$, $\eta_0=\nu_0\circ U_0$, and $\nu\circ U_0$ are mutually equivalent
	\item[(c)] By the complete  maximum principle (cf. e.g. \cite[(2.2) Proposition]{Ge 80}) , it is possible to choose the 
function $\varphi$   such that $U_0\varphi$ is bounded.  
Consequently, the measure $\lambda\circ U_0$ is $\sigma$-finite  for every finite measure $\lambda$
on $E$. Therefore, one could consider the measure $\nu\circ U_0$,
without normalizing constants $U_0\varphi(x_k)$, instead of $\eta_0$. This measure is   $\cu$-excessive 
and could have been equally used in what follows. 
However, we want to apply results from \cite{RoTr} and these are given w.r.t. the measure 
$\eta_0$. 
\end{itemize}

\end{rem}

\medskip
Define
\begin{eqnarray}\label{semigroup}
p_t^{\alpha}:=e^{-\alpha t}{p}_t; \ \ t\ge 0,\ \alpha\ge0.
\end{eqnarray}
It is known from \cite[Proposition 2.4]{RoTr} that there are unique extensions 
$({p}^{\beta}_t)_{t\ge 0}$ on $L^2(E;\eta)$ (resp. $({p}_t)_{t\ge 0}$ on $L^2(E;\eta_0)$ in the transient case)
as strongly continuous semigroups of contractions on the respective $L^2$-spaces. Moreover the adjoint semigroup 
of the respective extensions on the $L^2$-spaces are sub-Markovian (see \cite[Remark 3.1(ii)]{RoTr}). Using these 
extended semigroups one can define uniquely the $L^2$-generator which in turn determines a generalized Dirichlet 
form (see \cite[Section 3]{RoTr}). 
We shall denote these generalized Dirichlet forms by $(\ce^{\beta},\cf^{\beta})$, resp. $(\ce^0,\cf^0)$. 
So, in particular $\cu_{\beta}$ is associated to $(\ce^{\beta},\cf^{\beta})$, and  $\cu$ is associated to $(\ce^0,\cf^0)$ in the transient case.\\ \\
In the next lemma we will assume that the set $\{x_k;k\ge 1\}$ has a special form as follows:  
since the semi-Dirichlet form with process resolvent $\cv$ is quasi-regular there exists a 
nest $(F_k)_{k\ge 1}$ of compacts. Since $E$ is separable each $F_k$ is also separable.
Then choose for each $F_k$ a countable dense set $\{x^k_n;n\ge 1\}$, and a countable 
dense set $\{x^0_n;n\ge 1\}$ in $E$. Let 
$$
\{y_n; n\ge 1\}:=\bigcup_{k\ge 0}\{x^k_n;n\ge 1\}.
$$ 
$\{y_n; n\ge 1\}$ is a special countable and dense subset of $E$. We have the following:

\begin{lemma} \label{lemma3.8} % Lemma 3.9
Suppose that the absolute continuity condition holds.
\begin{itemize}	
	\item[(i)] It always holds $\xi \ll \eta \ll\mu$.
	\item[(ii)] Let $\nu$ be defined w.r.t. the special dense subset $\{y_n; n\ge 1\}$. 
	Then $\xi$, $\eta $, and $\mu$ are mutually equivalent measures. 
\end{itemize}
\end{lemma}

\begin{proof}
(i) Clearly $\xi \ll\eta$. If $\mu(N)=0$ then $V_{\beta} 1_N=0$ by the absolute continuity condition and thus
$U_\beta 1_N  = (I+Q_\beta)V_\beta 1_N=0$, and so $\eta(N)=\nu(U_\beta 1_N )=0$. Hence $\eta\ll \mu$.

(ii) Let $\xi(N)=0$. Then $V_{\beta}1_N(x_n^k)=0$ for all $n,k$.
Since $\mu$ is $\sigma$-finite we may assume that 
$1_N\in L^2(E,\mu)$, otherwise we choose  $D_l\nearrow E$ with $\mu(D_l)<\infty$ for any $l$ and show the following 
for $N\cap D_l$ and any $l$. By $\ca$-quasi-continuity of $V_{\beta}1_N$ there exists an $\ca$-nest $(E_k)_{k\ge 1}$ such that 
$V_{\beta}1_N$ is continuous on each $E_k$ hence on each $\overline{F}_k:=E_k\cap F_k$. 
Therefore by approximation $V_{\beta}1_N=0$ on $\bigcup_{k\ge 1} \overline{F}_k$. It follows that $V_{\beta}1_N=0$ 
$\ca$-quasi-everywhere and so, $\mu(N)=0$  since  $(V_\alpha)_{\alpha>0}$ is a $C_0$-resolvent on $L^2(E, \mu)$.
\end{proof}

The following result offers  the claimed example of quasi-regular generalized Dirichlet form obtained by perturbation with kernels. 
It is a corollary of Proposition \ref{capacitiesperturbation}.

\begin{cor}  \label{GDF2}   % Corollary 3.10
Let $\eta$, and in the transient case $\eta_0$, be defined w.r.t. any dense subset 
$\{x_n; n\ge 1\}\subset E$, suppose that the absolute continuity condition holds, 
and that $Q_{\beta} 1\in \mbox{Qbd} (\cv_\beta)$. 
Consider the assumptions:
\begin{itemize}
	\item[(i)]  $\mu\circ P \ll \mu$.
	\item[(ii)] The measure $\mu \circ Q_{\beta}$ charges no  $\ca$-exceptional set. 
\end{itemize}
Suppose that either (i) or (ii) holds. Then the generalized Dirichlet form $(\ce^{\beta},\cf^{\beta})$ associated with 
$\cu_{\beta}$ on $L^2(E, \eta)$ is quasi-regular. 
If ${\bf (T)}$ holds  then  the generalized Dirichlet form $(\ce^{0},\cf^{0})$ associated with 
$\cu$ on $L^2(E, \eta_0)$ is quasi-regular. Moreover in either case the corresponding process is properly 
associated in the resolvent sense with the corresponding generalized Dirichlet form. 
\end{cor}

\begin{proof}
By Lemma \ref{lemma3.8} we have $\eta\ll\mu$ and by Remark \ref{abscont}(b) $\eta$ is equivalent to $\eta_0$ if ${\bf (T)}$ holds, 
hence we also have  $\eta_0\ll\mu$ if ${\bf (T)}$ holds. We have $\mu\circ P_{\beta} \ll\mu\circ P$, and $\ca$-exceptional sets are 
 $\ca_{\beta}$-exceptional. 
According to assertions  $(ii)$ and $(iii)$ of Proposition \ref{capacitiesperturbation} applied to $\cv_\beta$ instead of $\cv$
(this is possible taking into account assertions (c) and (d) of Remark \ref{subordinate}; in particular, from the hypothesis 
on $Q_\beta 1$ we get $Q_{2\beta}  1\leq Q_\beta 1 \in \mbox{Qbd} (\cv_\beta)\subset \mbox{Qbd} (\cv_{2\beta})$ 
and thus $Q_{2\beta}  1 \in \mbox{Qbd} (\cv_{2\beta})$), 
the processes corresponding to $\cu_\beta$  are  $\mu$-standard, hence  $\eta$-standard since $\eta\ll\mu$.
From Remark \ref{quasileftandcap}(c)  they are also $\eta$-tight. The quasi-regularity as well as the proper association with the forms $(\ce^{\beta},\cf^{\beta})$ follows 
with the help of Theorem \ref{thm2.4}(ii) as in the proof of Theorem \ref{gdfstandard}(i) and (ii). The transient case is similar.
\end{proof}

\vspace{3mm}

\noindent 
{\bf Acknowledgments}\\

{\small The first named author gratefully acknowledges support from the Romanian Ministry
of Education, Research, Youth and Sport (CNCSIS PCCE-55/2008).

The second named author gratefully acknowledges support from the research project "Advanced Research and Education of Financial Mathematics" 
at Seoul National University.}

\end{document}